\documentclass[12pt]{amsart}

\usepackage{verbatim,amssymb,latexsym,amscd}
\usepackage[all]{xy}
\usepackage[colorlinks,linkcolor=blue,citecolor=blue,urlcolor=red]{hyperref}
\usepackage[T1]{fontenc}
\usepackage{bbold}

\newcommand{\ie}{{\it i.e. }}
\newcommand{\cf}{{\it cf. }}
\newcommand{\eg}{{\it e.g. }}

\newcommand{\resp}{{\it resp. }}
\newcommand{\un}{\mathbb{1}}
\newcommand{\A}{\mathbf{A}}

\newcommand{\C}{\mathbb{C}}

\renewcommand{\L}{\mathbb{L}}

\newcommand{\N}{\mathbb{N}}
\renewcommand{\P}{\mathbf{P}}
\newcommand{\Q}{\mathbb{Q}}

\newcommand{\Z}{\mathbb{Z}}
\newcommand{\sA}{\mathcal{A}}
\newcommand{\sB}{\mathcal{B}}
\newcommand{\sC}{\mathcal{C}}

\newcommand{\sI}{\mathcal{I}}

\newcommand{\sK}{\mathcal{K}}

\newcommand{\sM}{\mathcal{M}}

\newcommand{\sS}{\mathcal{S}}

\newcommand{\sW}{\mathcal{W}}

\newcommand{\uH}{\underline{H}}

\newcommand{\Spec}{\operatorname{Spec}}

\newcommand{\Tr}{\operatorname{Tr}}
\newcommand{\Ker}{\operatorname{Ker}}

\newcommand{\add}{{\operatorname{add}}}

\newcommand{\Hom}{\operatorname{Hom}}
\newcommand{\End}{\operatorname{End}}

\newcommand{\Alb}{\operatorname{Alb}}

\newcommand{\Ex}{{\operatorname{\bf Ex}}}

\newcommand{\car}{\operatorname{char}}

\newcommand{\rat}{{\operatorname{rat}}}

\newcommand{\hun}{{\operatorname{hun}}}
\newcommand{\hum}{{\operatorname{hum}}}
\newcommand{\num}{{\operatorname{num}}}

\newcommand{\new}{{\operatorname{new}}}
\newcommand{\naf}{{\operatorname{naf}}}
\newcommand{\rig}{{\operatorname{rig}}}

\newcommand{\tnil}{{\operatorname{tnil}}}

\newcommand{\Corr}{\operatorname{Corr}}

\newcommand{\ab}{{\operatorname{ab}}}

\newcommand{\weil}{\sW}

\newcommand{\eff}{{\text{\rm eff}}}

\newcommand{\by}{\xrightarrow}

\newcommand{\iso}{\by{\sim}}

\renewcommand{\lim}{\varprojlim}

\renewcommand{\phi}{\varphi}
\renewcommand{\epsilon}{\varepsilon}

\newcounter{spec}
\newenvironment{thlist}{\begin{list}{\rm{(\roman{spec})}}%
{\usecounter{spec}\labelwidth=20pt\itemindent=0pt\labelsep=10pt}}%
{\end{list}}%

\numberwithin{equation}{section}

\newtheorem{thm}{Theorem}[section]
\newtheorem{prop}[thm]{Proposition}
\newtheorem{cor}[thm]{Corollary}
\newtheorem{conj}[thm]{Hypothesis}
\theoremstyle{definition}
\newtheorem{nota}[thm]{Notation}

\newtheorem{rk}[thm]{Remark}
\newtheorem{ex}[thm]{Example}

\begin{document}

\title{motives}
\author{Luca Barbieri-Viale}

\address{Dipartimento di Matematica ``F. Enriques'', Universit{\`a} degli Studi di Milano\\ Via C. Saldini, 50\\ I-20133 Milano\\ Italy}
\email{luca.barbieri-viale@unimi.it}
%\date{\today} 
\begin{abstract} 
\begin{sloppypar} 
Making a survey of recent constructions of universal cohomologies we suggest a new framework for a theory of motives in algebraic geometry.
\end{sloppypar}
\end{abstract}
\maketitle

\hfill {\small \it Dedicated to Jaap Murre}\\

\section*{Introduction}
The aim of this note is to relate the existing universal cohomology with the conjectural ``theory of motives''. Actually, the universal problems arising from the classification of cohomology theories do have canonical solutions, providing universal cohomologies with values in abelian categories: see \cite{BV} for the general setting, \cite{BVT} for the universal ordinary homology of topological spaces and \cite{BVKW} for the universal Weil cohomology. In particular, K\"unneth formulas and tensor structures are granted by universal constructions, see \cite{BVHP}, \cite{BVP} and \cite{BVK}.  We certainly should investigate more deeply into these universal theories but the general philosophical question is now the following: should we expect some key properties to be satisfied by universal cohomology theories in such a way that we can finally talk of them as a ``theory of motives''? 

For example, ``pure motives'' are usually constructed out of algebraic cycles and provide a  universal (generalised) Weil cohomology after Grothendieck standard conjectures; see Murre's lectures \cite{JPL} for a comprehensive presentation of classical Chow and Grothendieck motives. Notably, the construction of the universal (generalised) Weil cohomology is possible and independent of the standard conjectures, as \cite{BVKW} is showing. Therefore,  we may argue the other way around: we can look at this latter universal theory as the natural candidate out of which we may obtain a ``theory of pure motives''. Is this weaker approach reasonable? What we expect from a ``theory of pure motives''?

Note that motivic homotopy theory also appears to be a way to get around the standard conjectures if the wish is a triangulated (or dg or $\infty$) category as originally proposed by Voevodsky \cite{V}, Levine \cite{L} and Hanamura \cite{HA}. In Quillen's abstract homotopy setting of model categories Dugger's universal homotopy theories \cite{DU} provide the general context. However, similarly, we have a universal theory but the lack of conservativity of realisations is a missing wish, the most basic one, according to Ayoub \cite{AC}. 

Eventually, Voevodsky \cite[Def.\,4.3]{VN} provided a definition of what he called a ``theory of mixed motives''  but its existence would again imply the standard conjectures (see also Beilinson \cite{BE}). Furthermore, one should also believe that Voevodsky's nilpotence conjecture \cite[Conj.\,4.2]{VN} is true, because it also follows from his theory.\footnote{Recall that the nilpotence conjecture says that a cycle on a smooth projective variety is numerically equivalent to zero if and only if it is smash nilpotent.} 

Especially, Nori (see \cite{HMN} and \cite{BVHP}) and Andr\'e (see \cite{A}) provided an unconditional theory in characteristic zero. However, these theories are only universal for cohomology theories which are comparable with singular cohomology. 

In general, for a ``theory of motives'' relative to a chosen class of cohomology theories, a minimal key requirement is that of sharing a common universal enrichment: this is explained by Propositions \ref{prop1} - \ref{prop3} and Theorem \ref{thm4} below, after Theorems \ref{thm1} - \ref{thm2} and \ref{thm3} providing the universal theories (see also \cite[Prop.\,3.2.8]{BV} and \cite[Thm.\,6.1.7 - 8.4.5]{BVKW}). When this happens for a class of (generalised) Weil cohomologies we may say that a ``theory of pure motives'' exists for this class. The link with the standard conjectures is the content of Proposition \ref{prop4}, Corollary \ref{cor1} and Theorem \ref{thm5}. The nilpotence conjecture is independent of specific cohomology theories and implies all standard conjectures for any (generalised) Weil cohomology satisfying strong and weak Lefschetz as we here explain with Corollary \ref{cor2} (see also \cite[Cor.\,8.6.6]{BVKW}). 
All these considerations are leading to the Hypothesis \ref{SH} along with its direct consequences in Proposition \ref{prop5}: under this hypothesis, the universal Weil cohomology provides a Tannakian category such that any classical Weil cohomology yields a fiber functor in such a way that the conjectural picture for ``pure motives'' (as reported by Serre \cite[\S 3]{SE} or Deligne \cite[\S 1.1]{DE}) is verified. See the final Remark \ref{rk1} for a detailed explanation. Similarly, a ``theory of mixed motives'' via universal theories shall be properly treated elsewhere. 
 
In conclusion, from the existence of a universal cohomology we gain perspective that will help us see a <<new theory>> -- the ``theory of motives'' whether mixed or not -- <<that would finally explain the (similar) behaviour of all the different cohomology theories>> paraphrasing Grothendieck words, as reported by Jaap Murre.\footnote{See the footnote 1 in the introduction of \cite{JPL}}

\section{Exposition of the main theme}

Fix a category $\sC$ of  spaces or varieties.  To set a ``theory of motives'' on $\sC$ we need to select a class $\sS$ of interesting invariants.  Originally, for  the category $\sC$ of smooth  projective  algebraic varieties over a field $k$,  Grothendieck considered $\sS$ to be the class given by $\ell$-adic cohomologies in any characteristic with $\car k\neq \ell$, Betti and de Rham cohomologies in zero characteristic, and crystalline cohomology in positive characteristics (\eg see \cite[Ex.\,1.2.14]{JPL} or \cite[Def.\,4.3.2]{BVKW}):  call this latter the \emph{Grothendieck class} for short. By the way, this class gives rise to a class of relative cohomologies (in the sense of \cite[Def.\,3.2.1]{BV}). 

\subsection*{Universal cohomology} In general, for any  class $\sS$, we may assume that if $H\in \sS$ then for $X\in \sC$ and $i\in \N$ we have the cohomology objects $H^i(X)$ belonging to a fixed abelian category $\sA$, \eg finite dimensional $K$-vector spaces for $K$ a characteristic zero field of coefficients. Also, as $X\in \sC$ varies, these cohomology objects give rise to  a $\N$-indexed family $\{H^i\}_{i\in \N}$ of contravariant functors $H^i:\sC\to \sA$. 

Say that $H$ is a \emph{cohomology} with values in $\sA$ and note that for any (exact) functor between abelian categories $G: \sA \to \sB$ the composition with $H^i$ yields a cohomology with values in $\sB$: call this a \emph{push-forward} along $G$, denoted $G_*H$ for short. Moreover, such a family $\{H^i\}_{i\in \N}$ is equivalent to a single functor from $\sC\times \N$ to $\sA$ where $\N$ is considered as a discrete category.  Similarly, for $\sC^\square$ a suitable category of pairs, we have a basic notion of \emph{relative cohomology} on $\sC^\square$, see \cite[Def.\,3.2.1]{BV} for a precise formulation. 

Independently of a choice of $\sS$, as $\sA$ varies in $\Ex$, \ie  the 2-category of abelian categories and exact functors, we can see that the induced 2-functor of cohomologies with values in $\sA$ is already 2-representable, by the free abelian category on $\sC\times \N$, and similarly for the case of relative cohomologies. We obtain: 
\begin{thm}\label{thm1} For a category $\sC$ and a commutative ring $R$ the universal cohomology  $U$ exists, taking values in an abelian $R$-linear category $\sA(\sC)$. Moreover, for $\sC^\square$
the universal relative cohomology $U_\partial$ exists, taking values in the abelian $R$-linear  category $\sA_\partial (\sC)$. 
\end{thm}
See \cite[Thm.\,3.1.2, Cor.\,3.1.5, Rk.\,3.1.4, Thm.\,3.2.4 \& Cor.\,3.2.7]{BV} for details and notation. Therefore, for $H\in \sS$ with values in $\sA$ we get (uniquely up to isomorphisms) an exact functor 
  $$F_{H}: \sA(\sC)\to \sA$$ such that $H$ is the push-forward of $U$ along $F_{H}$; here $U$ 
 is given by $\{U^i\}_{i\in \N}$ where $U^i: \sC\to \sA(\sC)$ and we have that $H^i(X) \cong F_{H}(U^i(X))$ for each $X \in \sC$.  
 
Note that we have a canonical exact functor $\sA (\sC)\to \sA_\partial (\sC)$ which is not faithful, in general, but whose essential image is a generating subcategory (see  \cite[Thm.\,3.3.1]{BV}). If $H$ is a relative cohomology the functor $F_H$, lifts to an exact functor from $\sA_\partial (\sC)$ to $\sA$  which we denote $F_H^\partial$ , as depicted 
$$\xymatrix{
\sA (\sC)\ar[d]\ar@/^1.5pc/[dr]^-{F_H}&\\ 
\sA_\partial  (\sC)\ar@{.>}[r]^-{F_H^\partial}& \sA }$$
and we have $(F_H^\partial)_*(U_\partial)=H$. 
 
\subsection*{Equivalent cohomologies} For $H\in\sS$, consider the Serre quotient $$\sA(H):= \sA(\sC)/\Ker F_H$$ and, for $H$ relative, the Serre quotient $$\sA_\partial (H):=\sA_\partial(\sC)/\Ker F_H^\partial$$   (\cf \cite[Prop.\,3.2.8]{BV} and  \cite[Thm.\,6.1.7]{BVKW}). Clearly, $\sA(H)$ can be trivial, \eg if $H$ is the trivial theory. 

Say that $H, H'\in \sS$ with values in $\sA$ and $\sA'$ respectively are \emph{equivalent} if there is an equivalence $\sA(H)\cong \sA(H')$, \ie $\Ker F_H = \Ker F_{H'}$. Similarly, in the relative case, requiring $\sA_\partial(H)\cong \sA_\partial(H')$, and we then say that are \emph{$\partial$-equivalent}.

Moreover, say that $H'$ is a \emph{realisation} of $H$ if there exists a faithful\footnote{We can similarly formulate ``appropriate realisations'' in other contexts, rather than the abelian one, requiring conservativity instead of faithfullness, as in \cite{AC}.}
 exact functor $G: \sA \to \sA'$ such that $H'=G_*H$ is the push-forward of $H$ along $G$. In this case, conversely, we also say that $H$ is an \emph{enrichment} of $H'$ (\cf \cite[Def.\,6.1.4]{BVKW}).

The push-forward $U_H$  of $U$  along the projection to $\sA(H)$ yields $H$ as a realisation of $U_H$. Similarly, in the relative case, we get $U_{\partial,H}$ from $U_\partial$, with values in  $\sA_\partial (H)$.  The universal property of Serre quotients implies that $U_H$ is universal with respect to this property, which is Nori's universal property in the case of $U_{\partial,H}$  (\cf \cite[Thm.\,7.1.13]{HMN}).
Say that $U_H$ with values in  $\sA (H)$ is the \emph{initial or universal enrichment} of $H$ and $U_{\partial,H}$ with values in  $\sA_\partial (H)$ is the  \emph{universal $\partial$-enrichment} of a relative $H$.

Say that two theories are \emph{comparable} (or  \emph{$\partial$-comparable}) if they have common realisations, translating comparison theorems in this context. A comparison between $H$ and $H'$ is given by the existence of realisations  $G_*H$ and $G'_*H'$ along (faithful) functors $G: \sA \to \sB$ and $G': \sA' \to \sB$, together with a comparison isomorphism $G_*H\cong G'_*H'$ in $\sB$. Actually, we easily see that $H$ and $H'$ are equivalent (or $\partial$-equivalent) whenever they are comparable (or $\partial$-comparable) (\cf \cite[Def.\,6.6.1 \& Prop.\,6.6.2]{BVKW}). 

\subsection*{Effective motives} Now consider
$$\sI_\sS:= \bigcap_{H\in \sS}\Ker F_H  \ \
\text{and}\ \ \sA_\sS(\sC):= \sA(\sC)/\sI_\sS$$
so that we get an induced $U_{\sS}$, push-forward of $U$ along the projection to $\sA_\sS(\sC)$;  for $H\in \sS$, each $U_H$ is the push-forward of $U_{\sS}$ along the further quotient $\sA_\sS(\sC)\to \sA(H)$. There is a  functor $$r_H: \sA_\sS(\sC)\to \sA$$ refining $F_H$, so that $H$ is the push-forward of $U_\sS$  along $r_H$. Similarly define $\sI_{\partial,\sS}$ in the relative case, and $\sA_{\partial, \sS}(\sC)$ as $\sA_\partial (\sC)/\sI_{\partial,\sS}$, get $\sA_{\partial, \sS}(\sC)\to \sA_\partial (H)$, $U_{\partial, \sS}$ and $r_{\partial, H}$ as well. 
 We can easily express the similar behaviour the different cohomology theories in $\sS$ as follows:
\begin{prop} \label{prop1} The following are equivalent:
\begin{itemize}
\item all $H\in \sS$ are equivalent (\resp $\partial$-equivalent in the relative case),
\item all $H\in \sS$ are realisations of $U_\sS$ along $r_H$ (\resp of $U_{\partial, \sS}$ along $r_{\partial, H}$),
\item if $H\in \sS$ then $\sA_\sS(\sC)\iso\sA(H)$ (\resp $\sA_{\partial, \sS}(\sC)\iso\sA_\partial(H)$) is an equivalence.
\end{itemize}
\end{prop}

Say that the cohomology theory $U_\sS$ (or the relative cohomology $U_{\partial, \sS}$) with values in $\sA_\sS(\sC)$ (or $\sA_{\partial, \sS}(\sC)$) is a \emph{theory of effective motives} for $\sC$ (or $\sC^\square$) and $\sS$ if the properties of Proposition \ref{prop1} are satisfied.
\begin{ex} \label{ex1}
If  $\sC$ is the category of algebraic varieties over a field $k$, $\sC^\square$ being given by pairs $(X, Y)$ where $Y$ is a closed subvariety of $X$ and $\sS$ is the Grothendieck class as above, a theory of motives exists for $k\hookrightarrow\C$ a fixed embedding in the complex numbers: this is the (effective) theory of Nori motives (see \cite[Def.\,9.1.3]{HMN} and \cite[Prop.\,3.2.8]{BV}). This is however cheating a bit: all cohomologies in the Grothendieck class are $\partial$-comparable since we are in zero characteristic! It is an open problem if such a theory exists in positive characteristics; another general problem is to give a geometric presentation of effective motives, if any.
\end{ex}

\subsection*{More axioms} Moreover, given any such a class $\sS$ one usually classifies its elements by the properties they share. This approach shall be driving us to an axiomatic notion of cohomology theory on $\sC$ which has then to be related back with the collection of cohomologies in $\sS$. For a chosen set $\dag$  of axioms in which some canonical maps between cohomology objects are invertible plus compatibilities (provided by the commutativity of some diagrams) call \emph{$\dag$-decoration} and/or \emph{$\dag$-cohomology theory} on $\sC$ the resulting theory satisfying these axioms. 
 
Let $\sI^\dag\subset \sA (\sC)$ be the smallest Serre subcategory containing kernels and cokernels of the set of morphisms that we want to make invertible (including the equalisers for compatibilities) for a $\dag$-decoration. Then
$$\sA^\dag (\sC):= \sA(\sC)/\sI^\dag$$
is just the free abelian category generated by $\sC\times \N$ modulo the relations given by the axioms $\dag$. We get the \emph{universal $\dag$-cohomology} $U^\dag$ as the push-forward of $U$ along the projection. Similarly, if $\sI^\dag_\partial\subset \sA_\partial (\sC)$ in the relative case, $\sA^\dag_\partial (\sC):= \sA_\partial^\dag (\sC)/\sI^\dag_\partial$.

\begin{prop} \label{prop2}
If $\sS$ is the class of $\dag$-cohomologies on $\sC$, for a set $\dag$ of axioms,  then 
$\sI^\dag=\sI_\sS$ therefore $\sA^\dag (\sC) = \sA_\sS (\sC)$. Similarly, $\sI^\dag_\partial =\sI_{\partial, \sS}$ and $\sA^\dag_\partial (\sC)=\sA_{\partial, \sS} (\sC)$
\end{prop} 
\begin{proof} In fact, clearly $\sI^\dag\subseteq \sI_\sS$ by the universality of $\sA^\dag_\partial (\sC)$  but $\sI^\dag = \Ker F_{U^\dag}$, since $F_{U^\dag}$ is the projection from $\sA(\sC)$ to $\sA^\dag (\sC)$, whence also
$\sI_\sS\subseteq \sI^\dag$. Similarly, for relative cohomologies.
\end{proof}

For example, all $H\in \sS$ can be (finitely) additive, \ie
$H^i(X\coprod Y)\cong H^i(X)\oplus H^i(Y)$, if coproducts exist $X\coprod Y\in \sC$. To impose or ask for additivity we just impose or ask that the canonical arrow between $H^i(X\coprod Y)$ and $H^i(X)\oplus H^i(Y)$ is invertible (\cf \cite[Prop.\,3.3.9]{BV}). Then $\sI^\add$ shall be the Serre subcategory generated by kernels and cokernels of these canonical arrows. These $H\in \sS$ can certainly be more $\dag$-decorated, satisfying further properties given by saying that some other canonical morphisms are invertible or some compatibilities. However, if we take $\dag =\add$ and set $\sS$ to be exactly the class of additive cohomologies we get $\sA^\add (\sC) = \sA_\sS (\sC)$ as an instance of Proposition \ref{prop2}. 

For a  point axiom (see \cite{BVT}), a key axiom indeed, we require that $H^i(1)=0$ for $i\neq 0$ and $H^0(1)\in \sA$ shall be called the \emph{coefficient object} of the cohomology theory. The resulting $\sA^{\rm point} (\sC)$, the theory with the point axiom, is obtained by taking $\sI^{\rm point}$ to be the Serre subcategory generated by all $H^i(1)$ for $i\neq 0$.

As the reader can easily guess, we can also ask for other $\dag$-decorations, such as homotopy invariance and excision, see \cite{BVT}, where we got the universal ordinary (or Eilenberg-Steenrod) relative (co)homology for topological spaces or smooth schemes. 

\begin{ex}\label{ex2}
In the topological case of $\sC$ being the category of CW-complexes, the category $\sC^\square$ given by usual pairs, the ring $R=\Q$ in Theorem \ref{thm1}, $\dag$= Ord (ordinary additive homology), then $\sA^{\rm Ord}_\partial (\sC)$ is the category of $\Q$-vector spaces and the universal theory is singular (co)homology,  see \cite[Thm.\,3.2.2]{BVT}. A theory of effective motives exists in this case. However, for a general coherent ring $R$, singular (co)homology is a push-forward of the universal theory, see \cite[Cor.\,3.2.3]{BVT}.
\end{ex}

\section{Development and monodial interlude} 

On $\sC$ we usually have a tensor structure\footnote{Tensor structure is an abbreviation for unital symmetric monoidal structure.} $(\sC, \times, 1)$ given by  the product $X\times Y\in \sC$  of varieties $X, Y\in \sC$ and where $1$ is the final object of $\sC$. Similarly, for pairs we get $(\sC^\square, \times, (1,\emptyset))$ an induced tensor structure.
Cohomology theories $H\in\sS$, \eg for the Grothendieck class, actually take values in abelian tensor $\Q$-linear categories $(\sA, \otimes, \un)$ and are endowed with an external product
\[\kappa^{i,j}_{X,Y}:H^i(X)\otimes  H^j(Y)\to H^{i+j}(X\times Y)\]
which satisfies some compatibilties, such as graded commutativity with respect to the symmetry and associativity (\eg see \cite[Def.\,3.2.3]{BVKW} and \cite[\S 2.1]{BVP}). Moreover, $H$ satisfies K\"unneth formula, \ie the corresponding morphism
\[\kappa_{X,Y} =\sum \kappa^{i,j}_{X,Y}: \bigoplus_{i+j=k} H^i(X)\otimes  H^j(Y)\iso H^{k}(X\times Y)\]
is an isomorphism. They are also satisfying the point axiom  $H^i(1)= 0$ for $i\neq 0$ and $\upsilon : H^0(1)\cong \un$ which is strong unitality (in \cite[Def.\,3.2.1]{BVKW}). 

\subsection*{K\"unneth cohomologies} Call this extra structure $(H, \kappa, \upsilon)$ along with the named properties a \emph{K\"unneth cohomology} on $\sC$, for short. It is well known that these properties just say that the corresponding graded contravariant functor $H^*: \sC\to \sA^\N$ is a strong tensor functor (\eg see \cite[Rk.\,3.2.2 \& 3.2.4]{BVKW}). Adopt the definitions of \cite[\S 2.1]{BVP} and the resulting relative K\"unneth formula (\cf \cite[Thm.\,2.4.1]{HMN} and the corresponding one displayed in \cite[\S 2.3]{BVP}) to define a \emph{relative K\"unneth cohomology} on $\sC^\square$.  

For (relative) K\"unneth cohomologies, we then can make up a 2-functor  as $(\sA, \otimes, \un)$ is varying in $\Ex^\otimes$, \ie the 2-category of abelian tensor $\Q$-linear categories with exact tensor\footnote{We consider $(\sA, \otimes, \un)$ unital symmetric monoidal for which $\otimes$ is exact but we do not assume that the $\Q$-algebra of endomorphisms of the unit $\un$ is a field.} and exact strong tensor $\Q$-linear functors, by pushing forward along such functors.  Merging the techniques from \cite{BVHP}, \cite{BVP}, \cite{BVKW} and \cite{BVK}  we get that this 2-functor is 2-representable.

\begin{thm}\label{thm2}  
For $(\sC, \times, 1)$ there exists a universal K\"unneth cohomology $(U_\kappa, \kappa, \upsilon)$ with values in 
$(\sA_\kappa(\sC), \otimes_\kappa, \un_\kappa)$ abelian tensor $\Q$-linear category such that $\otimes_\kappa$ is exact. Moreover, for $(\sC^\square, \times, (1,\emptyset))$we get a universal relative K\"unneth cohomology theory $(U_{\partial\kappa}, \kappa, \upsilon)$  with values in an abelian tensor exact $\Q$-linear category $(\sA_{\partial \kappa} (\sC),\otimes_{\partial \kappa}, \un_{\partial \kappa} )$.  
\end{thm}
\begin{proof} We follow the pattern of the proof of \cite[Thm.\,5.2.1]{BVKW} but for the opposite category as we want a contravariant theory. We may assume that $\sC$ is $\Q$-linear and let $\sC \times \N$ be endowed with the induced tensor structure with the Koszul constraint (see \cite[Rk.\,3.2.4]{BVKW}). (If $\sC$ is not $\Q$-linear we can pass to the $\Q$-linear preadditive hull $\Q\sC$ or just use that $\sC \times \N$ is graded in the sense of \cite[Def.\,2.13]{BVHP}). We first use Levine's universal construction (see \cite[Prop.\,Part II I.2.4.4 (i)]{L})
$$ (L,\ell, \upsilon) : \sC\times \N\to (\sC\times \N)^\kappa$$ providing a tensor structure $((\sC\times \N)^\kappa, \otimes, \omega)$ together with a universal  external product 
\[\ell_{X,Y}^{i,j}: L(X,i)\otimes L(Y,j)\to L(X\times Y,i+j)\]
and $\omega\to L(1,0)$ for $X,Y\in \sC$, $i,j\in \N$. Let $(\sC\times \N)^{\kappa,\add}$ be the relative $\Q$-linear additive completion (see \cite[Prop.\,3.4.3]{BVKW}), in such a way that the functor $(\sC\times \N)^\kappa\to (\sC\times \N)^{\kappa,\add}$ is a strong tensor functor and its composition $L^\add$ with $L$ is $\Q$-linear. 
 We then get $\upsilon: \omega^\add\to L^\add(1,0)$ and
\[\delta= \sum\ell_{X,Y}^{i,j}: \bigoplus_{i+j=k} L^\add(X,i)\otimes L^\add(Y,j)\to L^\add(X\times Y,k)\]
 in $(\sC\times \N)^{\kappa,\add}$, for $i,j,k\in \N$. Let $\sK (\sC)$ be the ($\Q$-linear tensor) localisation of  $(\sC\times \N)^{\kappa,\add}$ making $\upsilon$, $\delta$ and $L^\add(1,i)\to 0$ for $i> 0$ invertible (see \cite[Prop.\,3.5.3]{BVKW}). Now we get a K\"unneth cohomology $K$ with values in $\sK (\sC)$ which is a $\Q$-linear additive tensor category. Let $\sA_\kappa (\sC) := T (\sK (\sC))$, where $T$ is the 2-functor providing the universal abelian tensor category with an exact tensor structure, see \cite[Prop.\,5.4]{BVK}. It is easy to check, step by step, that any K\"unneth cohomology factors (uniquely up to ismorphism) through $\sA_\kappa (\sC)$.
 
For relative K\"unneth cohomologies, we modify the previous proof as follows: instead of $(\sC\times \N)^\kappa$ we consider the tensor category $(\sC^\square\times \N)^{\partial\kappa}$ providing a universal external product 
\[{ }^\square\ell^{i,j}: L(X, Y,i)\otimes L(X',Y',j)\to L(X\times X',X\times Y'\cup Y\times X', i+j)\]
togheter with ${ }^\square\upsilon: \omega\to L(1,\emptyset,0)$ and the boundaries 
\[ \partial^i: L(Y, Z,i) \to L(X, Y,i+1)\]
for any $(X, Y), (X',Y'), (Y, Z)\in \sC^\square$ and $i,j\in \N$. In fact, following Levine's construction (see \cite[Part II I.2.1.2]{L}), it is always possible to adjoin such morphisms $\{\partial^i\}_{i\in \N}$ in such a way that now $(\sC^\square\times \N)^{\partial\kappa}$  is also universal with respect to this property, \ie for functors $H$ into $\sA$ together with a choice of such boundary morphisms $\partial^i_H:H(Y, Z,i) \to H(X, Y,i+1)$. 
Note that the functoriality of $L$ grants excision maps 
\[ \Delta^i : L(X, Z,i)\to L(Y, Y\cap Z,i)\]
for $(X, Y), (X, Z)\in \sC^\square$ such that $X=Y\cup Z$. Now let $(\sC^\square\times \N)^{\partial\kappa, \add}$ be the relative additive completion in such a way that we also obtain the relative version ${ }^\square\delta= \sum{ }^\square\ell^{i,j}$ as above, given by ${ }^\square\ell^{i,j}$. Now to impose all the needed properties, we can directly consider the universal abelian tensor category  $T ((\sC^\square\times \N)^{\partial\kappa, \add})$ and get $\sA_{\partial \kappa} (\sC)$ by successive quotients modulo Serre tensor ideals (using \cite[Prop.\,4.5]{BVK}), since when the induced $U_{\partial \kappa}: \sC^\square \times \N\to \sA_{\partial \kappa} (\sC)$ becomes a relative K\"unneth cohomology $U_{\partial \kappa}^i(X,Y):= U_{\partial \kappa}(X,Y,i)$. We first impose that any triple gives rise to a complex via $L$ and then set the long exact sequence of a triple as in the proof of \cite[Thm.\,3.2.4]{BV}; we can make ${ }^\square\delta, { }^\square\upsilon$ and $\Delta^i$ invertible, impose that $U_{\partial \kappa}^i(1,\emptyset )=0$ for $i\neq 0$ and the various compatibilities, \eg  
${ }^\square\ell^{i,j}$ and $\partial^i$ become compatible in the sense of the axioms in \cite[\S 2.1]{BVP}. 
The so obtained relative K\"unneth cohomology $U_{\partial \kappa}$ is universal by construction. In fact, any relative K\"unneth cohomology $(H, \kappa, \upsilon)$ on $\sC^\square$ with values in $\sA$ yields a uniquely determined strict tensor functor $F_H^{\partial\kappa}:(\sC^\square\times \N)^{\partial\kappa}\to \sA$ such that $F_H^{\partial\kappa}\mid_{\sC^\square\times \N}= H$,  $F_H^{\partial\kappa}({ }^\square\ell^{i,j})= \kappa^{i,j}$, $F_H^{\partial\kappa}({ }^\square\upsilon)=\upsilon$ and $F_H^{\partial\kappa}(\partial^i)=\partial_H^i$. This $F_H^{\partial\kappa}$ further lifts to an exact tensor functor on $T ((\sC^\square\times \N)^{\partial\kappa, \add})$ and then factors trhough a functor $F_H^{\partial \kappa}: \sA_{\partial \kappa} (\sC)\to \sA$, uniquely, by construction.
\end{proof}

Clearly, we get an exact strong tensor functor $\sA_\kappa(\sC) \to \sA_{\partial \kappa}  (\sC)$.

\subsection*{Effective motives with tensor product} Therefore, from Theorem \ref{thm2},  for any K\"unneth cohomology $H$ with values in $\sA$ (abelian and tensor exact) we get an exact strong tensor $\Q$-linear functor $F_{H}^\kappa: \sA_\kappa (\sC)\to \sA$, lifting to $\sA_{\partial \kappa}  (\sC)$ if $H$ is relative. Thus $\Ker F_{H}^\kappa\subseteq \sA_\kappa (\sC)$ and $\Ker F_{H}^{\partial \kappa}\subseteq \sA_{\partial\kappa } (\sC)$, if $H$ is relative, are  Serre tensor ideals and the quotients $$\sA_\kappa (H): = \sA_\kappa (\sC)/\Ker F_H^\kappa\ \ \text{and}\ \ \sA_{\partial \kappa} (H):= \sA_{\partial \kappa}  (\sC)/\Ker F_H^{\partial \kappa}$$ inherit an exact tensor structure (see \cite[Prop.\,4.5]{BVK}). The push-forward $U_{\kappa, H}$  in $\sA_\kappa (H)$ is the \emph{universal $\kappa$-enrichment} of a K\"unneth cohomology $H$ and $U_{\partial \kappa, H}$  in $\sA_{\partial \kappa} (H)$ is the \emph{universal $\partial\kappa$-enrichment} of $H$ relative.  Say that $H$ and $H'$ are \emph{$\kappa$-equivalent} if we have a tensor equivalence $\sA_\kappa (H)\cong \sA_\kappa (H')$. Call relative theories  \emph{$\partial\kappa$-equivalent} if we have a tensor quivalence $\sA_{\partial \kappa} (H)\cong \sA_{\partial \kappa} (H')$. 

Say that a \emph{theory of (effective) motives with tensor product} exists for $\sC$ (or $\sC^\square$) and $\sS$ if the following properties, listed in Proposition \ref{prop3} below, are satisfied for the tensor $\kappa$-variant $U_{\kappa,\sS}$ and $r_{\kappa,H}$ (or $\partial\kappa$-variant $U_{\partial\kappa,\sS}$ and $r_{\partial\kappa,H}$) of Proposition \ref{prop1} above.

\begin{prop} \label{prop3} The following are equivalent:
\begin{itemize}
\item all $H\in \sS$ are $\kappa$-equivalent (\resp $\partial\kappa$-equivalent in the relative case),
\item all $H\in \sS$ are realisations of $U_{\kappa,\sS}$ along $r_{\kappa, H}$ (\resp of $U_{\partial\kappa, \sS}$ along $r_{\partial\kappa, H}$),
\item if $H\in \sS$ then $\sA_{\kappa, \sS}(\sC)\iso\sA_\kappa(H)$ (\resp $\sA_{\partial \kappa, \sS}(\sC)\iso\sA_{\partial \kappa}(H)$) is an equivalence.
\end{itemize}
\end{prop}

Summarizing up, for $H\in \sS$ relative K\"unneth cohomology, from the universal properties of Theorem \ref{thm1}, forgetting the tensor structures in Theorem \ref{thm2},  we get the following commutative diagram in $\Ex$
$$\xymatrix{\sA (\sC)\ar[r]  \ar[d] & \sA_\kappa (\sC)\ar[d]\ar@/^2.6pc/[ddr]^-{F_H^{\kappa}}&\\
\sA_\partial  (\sC)\ar[d]\ar[r] & \sA_{\partial \kappa} (\sC)\ar[d]\ar@/^1.1pc/[dr]^-{F_H^{\partial ,\kappa}}&\\ 
\sA_\partial  (H)\ar@/_1.3pc/[rr]\ar@{^{(}->}[r]^-{\gamma_H}& \sA_{\partial \kappa} (H)\ar@{^{(}->}[r]& \sA }$$
where $\gamma_H$ is also faithful: in some good cases $\gamma_H$ is a tensor equivalence.
\begin{ex} Keep considering $\sC^\square$ and the Grothendieck class $\sS$ from Example \ref{ex1} for $k\hookrightarrow\C$, regarding $H\in \sS$ as a relative K\"unneth cohomology. 
From Nori's basic lemma (\cite[Thm.\,2.5.7]{HMN}) applied to $H$ (= Betti cohomology is sufficient) we get that $\sA_\partial  (H)$ is canonically endowed with a tensor structure making $U_{\partial, H}$ in $\sA_\partial  (H)$ the universal  $\partial\kappa$-enrichment of $H$. Actually, to get the tensor structure we can proceed as follows (see also \cite[Thm.\,2.3.3]{BVP} for a more general statement): we have $\sC^{\rm good}\subset \sC^\square$ the full subcategory of good pairs for which $H$ is cellular and $H^{\rm good}: \sC^{\rm good}\times \N\to \sA$ is then a strong tensor functor, whence $\sA_\partial  (H^{\rm good})$ has a canonical tensor structure by \cite[Prop.\,8.1.5]{HMN} (or \cite[Cor.\, 4.4 \& Thm.\,2.20]{BVHP}) and $\sA_\partial  (H^{\rm good})\cong \sA_\partial  (H)$ is an equivalence by \cite[Thm.\,9.2.22]{HMN}. Thus $\gamma_H$ is a tensor equivalence of abelian tensor categories by the universal property of $\sA_{\partial \kappa} (H)$. Therefore, a theory of (effective) motives with tensor product exists in this case: it coincides with Nori motives. 
\end{ex}

\section{Recapitulation and coda}
In the following, we consider a  class  $\sS$ of K\"unneth cohomologies on the tensor category $\sC$ which is a suitable category of smooth projective varieties. Such an $H\in \sS$ with values in $(\sA, \otimes, \un)\in \Ex^\otimes$ shall be endowed with more structures and properties. First of all, further assume that the cohomology objects $H^i(X)$ are dualisable for all $X\in\sC$, \ie $H$ is taking  values in the full subcategory $\sA_\rig\subseteq \sA$ of dualisable objects. The category $\sA_\rig$ is a rigid abelian tensor $\Q$-linear category (the tensor structure is automatically exact, see \cite[Prop.\,4.1]{BVK}) and we let $\Ex^\rig\subset \Ex^\otimes$ be the full 2-subcategory determined by rigid categories. 

We shall talk about ``twists'' by picking a \emph{Lefschetz object} $L$, \ie an invertible  object of $\sA$, and writing $A(i):= A\otimes L^{\otimes -i}$ for $i\in \Z$ and $A\in\sA$. Since we have a ``dimension function'' on $\sC$ we may assume that there is a \emph{trace isomorphism} for $\dim X=n$, \ie an isomorphism 
 $$\Tr_X:H^{2n}(X)(n)\iso \un$$
 in $\sA$, whenever $X$ is geometrically connected. The dual $H^i(X)^\vee$ shall then be identified with $H^{2n-i}(X)(n)$ via \emph{Poincar\'e duality} as usual (and explained in \cite[Rk.\,4.2.3 b)]{BVKW} in even more generality). For $X, Y\in \sC$ we then have that the $\Q$-vector space of graded morphisms from $H^*(X)$ to $H^*(Y)$ can actually be computed as follows:
$$ \sA^{\N} (H^*(X), H^*(Y))\cong \sA^{\N} (\un , H^*(X)^\vee\otimes H^*(Y)) \cong  $$ $$ \sA^{\N} (\un , H^{2n-*}(X)(n)\otimes H^*(Y))\cong \sA^{\N} (\un , H^{2n-*}(X\times Y)(n))\cong$$ $$ \sA (\un , H^{2n}(X\times Y)(n))$$ 
and elements of this lattter are often called \emph{homological correspondences}. 
An extension of $H^*$ to \emph{algebraic correspondences}, \ie to the $\Q$-linear additive tensor category $\Corr$ given by Chow correspondences (\eg see \cite[\S 2.1.]{JPL}), shall be given by the existence of a $\Q$-linear \emph{cycle class map} (which is functorial, compatible with K\"unneth and $\Tr_X$)
$$ c\ell^i : {\rm CH}^i(X)_\Q \to \sA (\un , H^{2i}(X)(i))$$
from the Chow group of codimension $i$ cycles on $X$. Since $H^*$ is a contravariant functor from $\sC$ to $\sA^{\N}$ we get that any morphism $f:Y\to X$, shall be regarded as an algebraic correspondence via the transpose of its graph $\Gamma_f^{t}\subseteq X\times Y$ and  $H^*(f)$ is $ c\ell^n (\Gamma_f^{t})$. 

Let $\sM_\rat^\eff$ be the pseudo-abelian completion of $\Corr$, and $\sM_\rat:= \sM_\rat^\eff[\L^{-1}]$ for  $\L:= h^2(\P^1_k)$ the effective Lefschetz object in $\sM_\rat^\eff$, if $\P^1_k\in\sC$, \ie  $\sM_\rat$ is the pseudo-abelian rigid tensor category of Chow motives modelled on $\sC$ (\eg see \cite[Def.\,4.1.3]{BVKW} and \cite[\S 3.1.]{JPL}). Considering the contravariant strict tensor functor
$$h : \sC \to \sM_\rat$$ we can see that such additional structures and properties provide a lifting of $H^*$ along $h$ in such a way that
$$\uH^*:\sM_\rat\to \sA^{(\Z)}$$
is a strong graded commutative functor where $\sA^{(\Z)}$ denotes the finitely supported graded symmetric monoidal category. These structures and properties are available for any $H$ in the Grothendieck class.

\subsection*{Weil cohomologies} Fix such a suitable tensor full subcategory $\sC$ of the category of smooth projective varieties over a field $k$ such that $\P^1_k\in \sC$ (following \cite[Def\/.\,4.1.1]{BVKW} for non closed fields $k$ we may assume that $\sC$ is a strongly admissible subcategory). Say that  $(H, \kappa, \upsilon, \Tr, c\ell)$ is a \emph{Weil cohomology} on $\sC$ with values in $\sA$ rigid $\Q$-linear, if it is a K\"unneth cohomology that further satisfies Poincar\'e duality and the usual compatibility axioms which we don't recall here (also see \cite[Def.\,4.2.1]{BVKW} for details). 

For Weil cohomologies we then have $\dag$-decorations as before. For example, $H$ is \emph{normalised} if $H^0(\pi_0 (X))\iso H^0(X)$ and \emph{Albanese invariant} if $H^1(\Alb (X))\iso H^1(X)$ are invertible, see \cite[Def.\,4.3.4 \& 8.2.1]{BVKW}. A key decoration is the \emph{strong Lefschetz} property, \ie the invertibility of the canonical map 
$$L^{i}: H^{n-i}(X)\by{\simeq}H^{n+i}(X)(i)$$
induced by a Lefschetz operator $L$ for $i\le n =\dim X$; also recall that $H$ verifies \emph{weak Lefschetz} if a connected smooth hyperplane section $j : Y \hookrightarrow X$ of $X$
connected yields an isomorphism $j^*: H^i(X)\iso H^i(Y)$ for $i\leq n-2$, see \cite[Def.\,8.3.1]{BVKW}. 

Say that a Weil cohomology is \emph{tight} if it satisfies strong and weak Lefschetz, Albanese invariance and it is normalised as in \cite[Def.\,8.3.4]{BVKW}.

For any such (decorated) Weil cohomology $H$ on $\sC$ with values in $\sA$,  the target $\Q$-linear category $\sA\in \Ex^\rig$ has to be considered jointly with a choice of a Lefschetz object $L\in \sA$, so we let $(\sA, L)\in \Ex^\rig_*$ be the notation for pointed categories. 

The 2-functor of  Weil cohomologies on $\Ex^\rig_*$ is then obtained by pushing forward along $G: \sA \to \sB$ exact strong tensor functors compatibly with Lefschetz objects, \ie together with an isomorphsim $G_*(L_\sA)\cong L_\sB$ of Lefschetz objects. The 2-functors of all Weil and tight Weil cohomologies are 2-representable: this has been proven in \cite[Thm.\,5.2.1, Cor.\,5.2.2, Thm.\,8.4.1 \& Thm.\,A.5.1]{BVKW}, where even an additive graded version is treated; we here provide a simplified statement and proof. 

\begin{thm}\label{thm3}  
For $(\sC, \times, 1)$ as above there exists a universal Weil cohomology $(U_w, \kappa_w, \upsilon_w, \Tr_w, c\ell_w)$ with values in 
$(\sA_w(\sC), \otimes_w, \un_w)$ abelian rigid tensor $\Q$-linear category together with a Lefschetz object $L_w\in \sA_w(\sC)$. 

Moreover, for any $\dag$-decoration we get a universal $\dag$-Weil cohomology $U_w^\dag$ in $\sA_w^\dag(\sC)$, a Serre quotient of $\sA_w(\sC)$; in particular, for $\dag$= tight denote $U_w^+$ with values in $\sA_w^+(\sC)$ the tight universal one.\footnote{In \cite{BVKW} the pairs $(\sA_w(\sC), U_w)$ and $(\sA_w^+(\sC), U_w^+)$ are denoted $(\sW_\ab, W_\ab)$ and $(\sW_\ab^+, W_\ab^+)$ respectively.}
\end{thm}
\begin{proof}
First observe that a Weil cohomology $(H, \kappa, \upsilon, \Tr, c\ell)$ with values in $(\sA, L)\in \Ex^\rig_*$ is equivalent to the functor  $\uH^*$ as already noted, jointly with an isomorphism $\Tr: H^2(\P^1_k)= \uH^2(\L)\iso L= \un (-1)$ such that (i) $\uH^*(\L)$ is concentrated in degree $2$, (ii) $\uH^*(\sM_\rat^\eff) \subset \sA^\N$, and
(iii) if $X$ is geometrically connected then $\un=\uH^0(1)\iso\uH^0(h(X))$ is an isomorphism, here $1= h(\Spec k)$.
This is essentially well known and proven in \cite[Prop.\,4.4.1]{BVKW} in this generality. Applying the arguments in the proof of Theorem \ref{thm2} to $\sM_\rat^\eff \times \N$ we obtain the relative additive completion of the universal external product $(\sM_\rat^\eff \times \N)^{\kappa, \add}$ and the corresponding further localisation $\sK_w(\sM_\rat^\eff )$ of $\sK (\sM_\rat^\eff )$ at $L^\add(\P^1_k,i)$ for $i\neq 0, 2$ and $L^\add(1 ,0)\allowbreak\to L^\add(h(X),0)$ for $X$ geometrically connected; thus, applying $T$, we obtain $U_w^\eff$ with values in 
$\sA_w^\eff(\sC):= T(\sK_w(\sM_\rat^\eff ))$. Finally, tensor invert $L_w:= U_w^\eff(\P^1_k, 2)$ getting $U_w$ and
$\sA_w(\sC):=\sA_w^\eff(\sC)[L_w^{-1}]$. We thus have constructed $U_w^*:\sM_\rat \to \sA_w(\sC)^{(\Z)}$ such that any Weil cohomology  $\uH^*$ lifts uniquely to an exact tensor functor  $F_H^w :\sA_w(\sC)\to \sA$ compatibly with Lefschetz objects. We are left with checking that $\sA_w(\sC)$ is rigid. Since $U_w$ is taking values in $\sA_w(\sC)_\rig$ and this latter is an abelian tensor full subcategory of $\sA_w(\sC)$
we then should have $\sA_w(\sC)_\rig=\sA_w(\sC)$ by universality. 

Imposing, strong and weak Lefschetz, Albanese invariance and normalisation shall provide $\sA_w^\dag(\sC)$ as a Serre tensor quotient of $\sA_w(\sC)$ and $U_w^\dag$ as the push-forward of $U_w$.
\end{proof}

For $H\in \sS$ with values in $\sA$, associated with $F_H^w :\sA_w(\sC)\to \sA$, the exact strong tensor functor induced by Theorem \ref{thm3}, we get the $\Q$-linear abelian rigid tensor category
$$\sA_{w}(H) := \sA_w(\sC)/\Ker F_H^w$$ 
togheter with a Lefschetz object $L_{w, H}$ induced by the projection. This is the \emph{universal $w$-enrichment} of a Weil cohomology $H$.\footnote{The category $\sA_{w}(H)$ is denoted $\weil_H^\ab$ in \cite[Thm.\,6.1.7]{BVKW} where we also call it ab-initial enrichment.} Note that this is automatically tight if $H$ is tight and for the universal tight cohomology  $H= U^+_w$  we get $\sA_{w}(U^+_w)= \sA_w^+(\sC)$. 
\begin{ex} \label{ex4}
Let $\A_K^\flat$ be the absolutely flat completion of the $K$-algebra of abstract $p$-adic periods $\A_K$ introduced by Ayoub \cite[\S 1.3]{AW} for a valued field $K$ of unequal characteristics $(p, 0)$ and  residual field $k$. 
Let $$H^*_{\naf}(X) := H^*(\Gamma_{\new}(X/K)\overset{L}{\otimes}_{\A_K} \A^\flat_K)$$ be the restriction to $X\in \sC$ (smooth projective  varieties) of Ayoub's cohomology theory in \cite[Thm.\,1.5]{AW}. This $H_{\naf}$ is a Weil cohomology with values in the category $\A^\flat_K-{\rm mod}$  of  finitely generated projective $\A^\flat_K$-modules: the key facts here are that $H^*_{\naf}$ becomes a K\"unneth cohomology and homotopy invariance implies that $H^*_{\naf}$ is a strong tensor functor on $\sM_\rat$. We obtain $F_{H_\naf}(U_w)= H_{\naf}$ for an exact tensor functor $F_{H_\naf}^w :\sA_w(\sC)\to \A^\flat_K-{\rm mod}$. The universal $w$-enrichment of $H_\naf$ with values in $\sA_{w}(H_\naf)$ provides an a priori smaller algebra of periods.  Ayoub's conjecture \cite[Conj.\,3.20]{AW} is that $\A_K$ is a domain. 
\end{ex}
We can say that two Weil cohomologies $H$ and $H'$ are \emph{$w$-equivalent} if $\sA_{w}(H)$ and $\sA_{w}(H')$ are tensor equivalent compatibly with their Lefschetz objects. We also have a corresponding notion of \emph{$w$-comparable} (see \cite[Def.\,6.6.1 b)]{BVKW} for details).  It is easy to see that for $H$ Weil the canonical strong tensor functor $\sA_{\kappa}(H)\to\sA_{w}(H)$ is inducung a tensor equivalence 
$$\sA_{\kappa}(H)[L_{\kappa,H}^{-1}]\cong \sA_{w}(H)$$
where $L_{\kappa, H}$ is the Lefschetz object induced by $U_\kappa^{2}(\P^1_k)$ on $\sA_{\kappa}(H)$ after Theorem \ref{thm2} applied to $\sM_\rat^\eff $. 
\subsection*{Pure motives} Let $\sS$ be now a class of Weil cohomologies on $\sC$. The Grothendieck class is a class of tight Weil cohomologies,  nowdays called \emph{classical} (see \cite[Ex.\,1.2.14]{JPL} or \cite[Def.\,4.3.2]{BVKW}). Let $\sI_{\sS}^w \subset \sA_w(\sC)$ be the Serre tensor ideal $$\sI_{\sS}^w:= \bigcap_{H\in\sS} \Ker F_H^w \ \
\text{and}\ \ \sA_{w,\sS}(\sC) := \sA_w(\sC)/\sI_{\sS}^w$$ 
 the abelian rigid tensor category given by the Serre quotient along with the Weil cohomology theory $U_{w, \sS}$ push-forward of $U_w$. There is a further canonical  quotient 
$\sA_{w,\sS}(\sC)\to \sA_{w}(H)$ for all $H\in \sS$ making $U_{w, \sS}$ universal with respect to the class $\sS$. 
Let $r_{w,H} : \sA_{w,\sS}(\sC)\to \sA$ be the induced exact tensor functor.  
\begin{nota}\label{nota} Denote $Z(\sA):=\End_\sA (\un)$.\footnote{Recall key facts which are valid  for any abelian rigid tensor category $(\sA, \otimes, \un)$: we have that $Z(\sA)$ is absolutely flat, if it is a domain is a field, and, in this case, exact tensor functors $G:\sA\to \sB$ in $\Ex^\rig$ are faithful, see \cite[Lemma 2.3.1]{BVKW} for references and more details.}  Let $Z_\sS$  denote the $\Q$-algebra $\End_{\sA_{w,\sS}} (\un)$ for short. Let $Z$ (\resp $Z_+$) be $Z_\sS$ for $\sS$ the class of all Weil cohomologies (\resp  tight Weil cohomologies).
\end{nota}
Actually, we easily obtain the following analogue of Propositions \ref{prop1} - \ref{prop3}  (see \cite[Thm.\,6.6.3]{BVKW}). 
\begin{thm}\label{thm4}  
For $\sS$ containing the Grothendieck class, the following conditions are equivalent:
\begin{thlist}
\item $\sA_{w,\sS}(\sC)$ is connected, \ie $Z_\sS$ is a domain
\item  all $H\in \sS$ are realisations of $U_{w,\sS}$ along $r_{w, H}$ 
\item  all $H\in\sS$ are $w$-equivalent
\item if $H\in\sS$ then $\sA_{w,\sS}(\sC)\iso \sA_{w}(H)$ is a tensor equivalence
\item $\sA_{w,\sS}(\sC)$ is Tannakian.\footnote{In the sense of \cite{DT}.}
\end{thlist}
\end{thm}

Say that a \emph{theory of pure motives} exists for $\sS$ and $\sC$ if 
$\sA_{w,\sS}(\sC)$ is connected.
\begin{ex}\label{ex3} Let $\sS$ be exactly the Grothendieck class (= classical Weil cohomologies). For any field $k$ and $H\in\sS$, we have that $\sA_{w}(H)$ is Andr\'e's category whenever this latter is abelian, see \cite[Thm.\,9.3.3]{BVKW}. We have a canonical faithful exact tensor functor from $\sA_{w}(H)$ to Andr\'e's category, if it is abelian, which is also full, thus essentially surjective, see also \cite[Prop.\,10.2.1]{HMN} and \cite[Thm.\,6.5.1]{BVKW}.

If $k\hookrightarrow\C$ then Andr\'e's category is abelian semisimple (see \cite[Thm.\,0.4]{A}) and $\sA_{w,\sS}(\sC)\iso\sA_{w}(H)$, for any such $H$ in the Grothendieck class since all $H\in \sS$ are $w$-comparable. Therefore $Z_\sS =\Q$ and $\sA_{w,\sS}(\sC)$ is connected (Notation \ref{nota} and Theorem \ref{thm4}), it coincides with Andr\'e motives, \ie with $\sA_{w}(H)$ which is actually independent of $H$, and a theory of pure motives exists in this case: it coincides with Andr\'e motives.

However, a theory of pure motives for the Grothendieck class is missing in positive characteristics.
\end{ex}

\subsection*{Homological equivalence}
For any Weil cohomology $H$ regarded as a tensor functor  $\uH^*:\sM_\rat\to \sA^{(\Z)}$  there is an adequate \emph{$H$-homological equivalence} relation $\sim_H$ corresponding to the tensor ideal $\Ker \uH^*$. Let $\sM_H$ be the category $\sM_\rat$ modulo $\sim_H$, \ie the pseudo-abelian completion of $\sM_\rat/\Ker \uH^*$, as usual. We have  $$h_H: \sC \to \sM_H$$ induced by composition with $h$.  
Composing the induced faithful  functor  $\uH^*: \sM_H\hookrightarrow \sA^{(\Z)}$  with the direct sum functor, we get a refinement 
$$\uH:\sM_H\hookrightarrow \sA_w(H)$$
where $\uH (X):= \uH (h_H(X))= \oplus_i H^i(X)$ regarded in $\sA$ and  $\uH$ is faithful monoidal but not symmetric. The K\"unneth projectors
$$\pi^i :\uH (X) \twoheadrightarrow H^i(X) \rightarrowtail  \uH (X) $$
are given as usual, by the composition of the canonical projection and the canonical inclusion. Say that \emph{$\pi^i$ is algebraic} if it is given by an algebraic correspondence, \ie if it is in the image of $\uH$.

If $H$ is tight then the inverse of the Lefschetz operator $L^i$ induces, for $i\leq n = \dim X$, an homological correspondence as follows: denoting $\Lambda^{i}$ the $(-i)$-twisted inverse $(L^i)^{-1}$ we get an element of 
$$\sA (H^{n+i}(X), H^{n-i}(X)(-i))\cong \sA (\un ,H^{n+i}(X)^\vee\otimes H^{n-i}(X)(-i))\cong  $$ $$\sA (\un ,H^{n-i}(X)(n)\otimes H^{n-i}(X)(-i))\subset \sA (\un ,H^{2(n-i)}(X\times X)(n-i))$$
this latter being a direct summand of the former.  Say that \emph{$\Lambda^i$ is algebraic} if it is given by an algebraic correspondence, \ie by an algebraic cycle in ${\rm CH}^{n-i}(X\times X)$ via the cycle class map $c\ell^{n-i}$. Note that the $i$-twisted $L^{n-2i}$ also induces a commutative square
$$\xymatrix{
\sA (\un, H^{2i}(X)(i))\ar[r]^-{L^{n-2i}(i)}_-\sim & \sA(\un, H^{2(n-i)}(X)(n-i))\\ 
A_H^{i}  (X)\ar@{^{(}->}[u]\ar@{^{(}->}[r]& A_H^{n-i}  (X)\ar@{^{(}->}[u] }$$
where $A_H^{j}  (X)$ is the image of the cycle class map $c\ell^j$ for $j\leq n$. If $\Lambda^{n-2i}= c\ell^{2i}(\gamma)$ is algebraic, for a cycle $\gamma\in {\rm CH}^{2i}(X\times X)$, its action $\gamma_* : {\rm CH}^{n-i}(X)\to {\rm CH}^{i}(X)$  induces an isomorphism $A_H^{n-i}  (X)\iso A_H^{i}  (X)$.

Moreover, for the universal Weil cohomology $U_{w, \sS}$ relative to the class $\sS$, denote 
$$\sM_\sS:= \sM_{U_{w, \sS}}$$ given by an homological equivalence $\sim_\sS: = \sim_{U_{w, \sS}}$ finer than all such $H$-homological equivalences for $H\in\sS$. We obtain the following.
\begin{prop} \label{prop4}
For any $H\in\sS$ we have $\Q$-linear tensor functors
$$\sM_\tnil\to\sM_\sS\to \sM_H\to \sM_\num$$
where $\sM_\num$ and $\sM_\tnil$ are given by numerical equivalence and smash nilpotence equivalence (see \cite{J} and \cite{V} respectively). 
\end{prop}
\begin{proof}
Any Weil cohomology factors through smash nilpotence equivalence, any $H\in \sS$ is a push-forward of $U_{w, \sS}$ and numerical equivalence is coarser than any adequate equivalence relation (\cf \cite[Lemma 1.2.18]{JPL}). 
\end{proof}
By Proposition \ref{prop4} and Theorem \ref{thm4}, we get:
\begin{cor}\label{cor1}
A theory of pure motives for the class $\sS$ provides a unique homological equivalence $\sim_\sS=\sim_H$ for all $H\in\sS$. Voevodsky's nilpotence conjecture implies that $\sim_\tnil=\sim_\sS=\sim_H=\sim_\num$ is homological equivalence.
\end{cor}

In general, unconditionally, we get \emph{universal homological equivalence} relations; if $\sS$ is the class of all Weil cohomologies then $U_{w, \sS}= U_{w}\in \sS$ and 
$\sM_{U_w}= \sM_\sS$; if $\sS$ is the class of tight we have that $U_{w, \sS}=U_w^+$ and  $\sM_{U_w^+}=\sM_\sS$:  these coincide with the intersection of all $H$-homological equivalences for $H$ Weil and tight, respectively.\footnote{In \cite{BVKW} we denoted $\sim_\hun$ and  $\sim_\hum$, respectively, these universal homological equivalence relations.}

\subsection*{Standard hypothesis} Grothendieck philosophy of motives is suggesting that $\ell$-adic and $\ell'$-adic cohomology with $\ell\neq \ell'$ should be $w$-equivalent even in positive characteristics. This easily follows from Grothendieck standard conjectures \cite{GS} claiming the algebraicity of $\pi^i$ and $\Lambda^i$ jointly with the equality of $\ell$-adic homological equivalence with numerical equivalence; similarly, one expects that this should holds for any $H$ in the Grothendieck class, at least. 

Specifically, in our context of universal theories, the standard conjectures can be reformulated as follows (and as explained in \cite{BVKW}).  
\begin{thm}\label{thm5} For a Weil cohomology $H$ we have:
\begin{thlist}
\item $\uH:\sM_H\iso \sA_w(H)$ is an equivalence if and only if $\sM_H$ is abelian, $\uH$ is exact and the  $\pi^i$ are algebraic;
\item $\uH:\sM_H\iso \sA_w(H)$ is an equivalence and $\sA_w(H)$ is semi-simple if and only if $H$-homological equivalence is numerical equivalence and the  $\pi^i$ are algebraic. 
\end{thlist}
For $H$ tight we also have:
\begin{thlist}\stepcounter{spec}\stepcounter{spec}
\item $\uH:\sM_H\iso \sA_w(H)$ is an equivalence if and only if $\sM_H$ is abelian, $\uH$ is exact and the $\Lambda^{i}$ are algebraic; 
\item  $\uH:\sM_H\iso \sA_w(H)$ is an equivalence and $\sA_w(H)$ is semi-simple if and only if $H$-homological equivalence is numerical equivalence.
\end{thlist}
\end{thm}
\begin{proof} The proof of (i) and (ii) can be extracted from \cite[Thm.\,7.1.6 b) \& 7.2.5 b)]{BVKW} and that of (iii) and (iv) from 
\cite[Thm.\,8.6.10]{BVKW} but for the sake of exposition we explain the key points. If $\uH$ is an equivalence the  properties named in (i) are  clearly verified; conversely,  if $h_H(X)= \oplus h^i(X)$ is the K\"unneth decomposition  corresponding to the projectors $\pi^i$ we get that $h_H$ is a Weil cohomology with values in $\sM_H$ and, by the assumptions, it is a $w$-enrichment of $H$ so that $\uH$ is an equivalence by universality. Notably, $\sM_H $ is abelian semisimple if and only if $H$-homological equivalence is numerical equivalence, by  Jannsen \cite{J}. Thus (ii) follows from (i).

If $H$ is tight and $\uH$ is an equivalence the  properties named in (iii) are  clearly verified; conversely, the assumption implies the algebraicity of the $\pi^i$ so that also (iii) follows from (i). Finally, (iv) follows from 
(iii) since the algebraicty of the $\Lambda^{i}$ is granted by Smirnov \cite{S}. 
\end{proof}
In particular, as a consequence of Proposition \ref{prop4} and Theorem \ref{thm5} we have: 
\begin{cor}\label{cor2}
The standard conjectures for the universal cohomology $U_{w, \sS}$ of a class $\sS$ of Weil cohomologies imply equivalences 
$$\sM_\sS=\sM_H=\sM_\num \iso\sA_{w, \sS}(\sC)\iso \sA_{w}(H)$$ for $H\in \sS$. Voevodsky's nilpotence conjecture implies the standard conjectures for the universal cohomology $U_{w, \sS}$ of any class $\sS$ of tight Weil cohomologies and $ \sM_\tnil=\sM_\sS$ in the chain of the above equivalences. 
\end{cor}
Recall from  Notation \ref{nota} the absolutely flat ring $Z_+$ given by the endomorphisms of the unit of $\sA_{w}^+(\sC)$. 
\begin{conj}[\protect{Standard Hypothesis}]\label{SH}
The ring $Z_+$ is a domain.
\end{conj} 
Note that from Corollary \ref{cor2}, for any class $\sS$, the standard conjectures for $U_{w, \sS}$ imply that $Z_\sS = Z(\sM_\sS)=\Q$ (Notation \ref{nota}) and a theory of pure motives exists for $\sS$. Therefore, the nilpotent or standard conjectures implies the standard hypothesis with $Z_+=\Q$.

Moreover, Grothendieck pointed out (see \cite{GS} and  Kleiman's review \cite{KS}) that the formalism of Weil cohomologies jointly with the standard conjectures shall be sufficient to provide a natural proof of the Weil conjectures. The following result gives one reason why the standard hypothesis is important, akin to \cite[Cor.\,5-5]{KS}.

\begin{prop}\label{prop5} The standard hypothesis (Hypothesis \ref{SH}) implies that for every self correspondence $\alpha$ of $X\in\sC$ 
\begin{itemize}
\item   the trace $\Tr (\alpha|_{H^i(X)})\in Z_+$ and  the Lefschetz number
$$\sum (-1)^i \Tr (\alpha|_{H^i(X)})\in Z_+$$
are the same for all $H$ tight with values in $\sA$ abelian and rigid; in particular, the dimension of $H^i(X)$ and the Euler characteristic $\chi (X)\in Z_+$ is independent of the choice of a tight Weil cohomology;
 
 \item the characteristic polynomial $P_\alpha (t)$ of $\alpha|_{H^i(X)}$ has coefficients in $Z_+$ independently of $H$ tight with values in $K$-vector spaces.
  \end{itemize}
 \end{prop}
\begin{proof} If $H$ is tight with values in $\sA$ then $Z_+\hookrightarrow Z(\sA)$ (Notation \ref{nota}) and the action of $\alpha$ is in the image of $\End_{Z^+}(U_w^{+,i}(X))\hookrightarrow \End_{Z (\sA)}(H^i(X))$ compatibly with the traces. For $Z (\sA)= K$ the matrix representing the endomorphism $\alpha$ on  $H^i(X)$ has entries in $Z_+$ so that the characteristic polynomial $P_\alpha (t)\in Z_+ [t]$.\end{proof}

\begin{rk}\label{rk1} The standard hypothesis (Hypothesis \ref{SH}) is equivalent to the hypothesis that a theory of pure motives exists for tight Weil cohomologies. Actually, from Theorem \ref{thm4} for the class of tight Weil cohomologies, the standard hypothesis is equivalent to saying that $\sA_{w}^+(\sC)$ is Tannakian and $\sA_{w}^+(\sC) \iso \sA_w(H)$ for any $H$ tight which yields a fiber functor $r_{w, H}$  if $H$ is in the Grothendieck class. Moreover, the external $\Hom$ of $\sA_{w}^+(\sC)$ is finite dimensional over $Z^+$ and every object is of finite length, as for any Tannakian category (see \cite[Prop.\,2.13]{DT}). 

Therefore, under the standard hypothesis, the field $Z_+$ is playing the r\^ole of a ``universal field of coefficients'' since $Z_+$ is a subfield of $Z(\sA)$ for any $H$ tight with values in $\sA$ (see Notation \ref{nota}). In particular,  $\ell$-adic and $\ell’$-adic cohomologies, which are a priori not comparable, become equivalent, $Z_+$ is a subfield of $\Q_\ell$ for every $\ell$ and the field $Z_+$ could  actually be a large trascendental extension of $\Q$ for $k$ with positive characteristic. 

However, if $k\hookrightarrow \C$ then the standard hypothesis is that $Z_+=\Q$ because the universal theory is equivalent to singular cohomology; therefore $\sA_{w}^+(\sC)$ is Andr\'e's category and this latter category is then the universal enrichment of all tight Weil cohomologies.

By the way, for any field $k$,  the hypothesis that $Z$ (Notation \ref{nota}) is a domain is stronger, it implies that all Weil cohomologies are tight and $Z=Z_+$. 
Note that if the Weil cohomology of Example \ref{ex4} is tight then $Z_+$ is a subfield of $\A^\flat_K$ under the pertinent assumptions on $K$ and $k$. 
\end{rk}

\subsection*{Acknowledgement}
I would like to thank Joseph Ayoub, Bruno Kahn and Marc Levine for discussions on matters involving their work. This paper has been conceived after some lecture notes redacted during a visit at the Institute of Mathematical Sciences of Chennai (India) which I also thank for hospitality and support. I finally thank the Italian PRIN 2017 ``Geometric, algebraic and analytic methods in arithmetic'' for support.

\end{document}